\documentclass[11pt]{amsart}

\usepackage{amsmath,amssymb,amsthm,amsfonts,amscd,flafter,epsf, epsfig,graphicx,verbatim,pinlabel,mathrsfs}
\usepackage[all]{xy}
\usepackage{epsf} 
\usepackage{epstopdf}
\usepackage{palatino}
\usepackage[usenames,dvipsnames]{color}

\usepackage[colorlinks=true, urlcolor=NavyBlue, linkcolor=NavyBlue, citecolor=NavyBlue]{hyperref}

\title{Admissible transverse surgery does not preserve tightness}

\author[John A. Baldwin]{John A. Baldwin}
\address{Department of Mathematics \\ Boston College}
\email{john.baldwin@bc.edu}

\author[John B. Etnyre]{John B. Etnyre}
\address{Department of Mathematics \\ Georgia Institute of Technology}
\email{etnyre@math.gatech.edu}

\def\R{{\mathbb{R}}}
\def\C{{\mathbb{C}}}
\def\Z{{\mathbb{Z}}}

\newcommand\hf{\widehat{HF}}

\newcommand\Sc{\text{Spin}^c}
\newcommand\spc{\mathfrak{s}}

\newcommand\lS{\mathcal{S}}

\newcommand\slope{t}

\def\dfn#1{{\em #1}}
\DeclareMathOperator{\tb}{tb}

\newtheorem{theorem}{Theorem}[section]
\newtheorem{lemma}[theorem]{Lemma}

\newtheorem{corollary}[theorem]{Corollary}
\newtheorem{proposition}[theorem]{Proposition}
\newtheorem{question}[theorem]{Question}

\theoremstyle{definition}

\newtheorem{remark}[theorem]{Remark}

\begin{document}
\begin{abstract} 
We produce the first examples of closed, tight contact 3-manifolds which become overtwisted after performing admissible transverse surgeries. Along the way, we clarify the relationship between admissible transverse surgery and Legendrian surgery. We use this clarification to study a new invariant of transverse knots -- namely, the range of slopes on which admissible transverse surgery preserves tightness -- and to provide some new examples of knot types which are not uniformly thick. 
Our examples also illuminate several interesting new phenomena, including the existence of hyperbolic, universally tight contact 3-manifolds whose Heegaard Floer contact invariants vanish (and which are not weakly fillable); and the existence of open books with arbitrarily high fractional Dehn twist coefficients whose compatible contact structures are not deformations of co-orientable taut foliations.
\end{abstract}

\maketitle

\section{Introduction}
\label{sec:intro}

A longstanding and fundamental open question in 3-dimensional contact geometry asks whether the result of Legendrian surgery on a Legendrian knot in a closed, tight contact 3-manifold is necessarily tight; in other words,

\begin{question}
\label{ques:legtight}
Does Legendrian surgery preserve tightness for closed manifolds?
\end{question}  

It is worth noting that Legendrian surgery does \emph{not} always preserve tightness for non-closed manifolds; indeed, Honda has found a Legendrian knot in a tight genus four handlebody on which Legendrian surgery is overtwisted \cite{honda3}. For closed manifolds, however, it is well-known that Legendrian surgery preserves symplectic and Stein fillability and the property of having a non-zero Heegaard Floer contact invariant. This fact, together with the apparent difficulty of embedding Honda's example into a closed, tight contact manifold (no progress has been made), hints that the answer to Question \ref{ques:legtight} is ``yes."

In this paper, we study the relationship between tightness and an operation defined by Gay \cite{gay} called \emph{admissible transverse surgery}. This operation is in some ways a transverse analogue of Legendrian surgery. In particular, we prove the following relationship between the two surgery operations (the first part of the theorem was also observed in \cite{gay}).

\begin{theorem}\label{thm:same}
Every Legendrian surgery on a Legendrian knot results in a contact manifold which can also be obtained via an admissible transverse surgery on the transversal push off of the Legendrian knot. Conversely, most (but not all) admissible transverse surgeries on a transverse knot result in contact manifolds that can also be obtained via Legendrian surgery on a Legendrian link that lies in a neighborhood of the transverse knot. 
\end{theorem}
In Section~\ref{sec:legtrans}, we will prove several results that make the above theorem more precise.
Our main result, however, concerns the following analogue of Question \ref{ques:legtight}, which has been open for some time.

\begin{question}
\label{ques:transtight}
Does admissible transverse surgery preserve tightness for closed manifolds?
\end{question}

Note that an affirmative answer to this question would also provide an affirmative answer to Question \ref{ques:legtight}, by Theorem \ref{thm:same}. In analogy with Honda's example, Colin has found a transverse knot in an \emph{open}, universally tight contact manifold on which admissible transverse surgery is overtwisted \cite{colin}.
However, given the close relationship between admissible transverse surgery and Legendrian surgery and the expectation that the latter preserves tightness for closed manifolds, one might expect that the former does as well. We show that this is not the case; in other words, the answer to Question \ref{ques:transtight} is ``no."

\begin{theorem}\label{thm:main}
There exist infinitely many closed, universally tight contact manifolds $(M,\xi)$ for which an admissible transverse surgery on some transverse knot $K\subset M$ is overtwisted.
\end{theorem}

Our examples illuminate several other interesting new phenomena as well, many of which rely on the connections between admissible transverse surgery and Legendrian surgery established in Section \ref{sec:legtrans}. Below, we describe the nature of our examples and outline the proof of Theorem \ref{thm:main}. We then we highlight some of these new phenomena.

\subsection{Capping Off, Our Examples and The Main Theorem}
\label{ssec:cappingintro}
 
We first recall the operation on open books known as \emph{capping off}; this operation plays a key role in many of the results in this paper. Let $(S,\phi)$ be an abstract open book with at least two binding components. By capping off one of the boundary components of $S$ with a disk, we obtain an open book $(\widehat S,\widehat \phi),$ where $\widehat\phi$ is the extension of $\phi$ to $\widehat S$ by the identity on this disk. 

The contact manifolds in our (infinitude of) examples are supported by genus one open books with two boundary components and monodromies that are freely isotopic to pseudo-Anosov maps ({\em cf.}\/ Section \ref{sec:examples}). These open books are constructed so that {\bf (A)} they support universally tight contact structures, and  {\bf (B)} the open books obtained by capping off one of their boundary components support overtwisted contact structures. This leads immediately to the following theorem, which answers a question posed in \cite[Question 1.4]{bald5}.

\begin{theorem}
\label{otcapoff}
There exists an open book $(S,\phi)$ supporting a universally tight contact structure such that the open book $(\widehat S,\widehat \phi)$ obtained by capping off some boundary component of $S$ supports an overtwisted contact structure.
\end{theorem}

In Section \ref{sec:cappingoff}, we prove that capping off is equivalent to an admissible transverse surgery.

\begin{theorem}\label{thm:cappingsurgery}
The contact manifold supported by $(\widehat S,\widehat \phi)$ is obtained from that supported by $(S,\phi)$ by admissible 0-surgery on the binding component of $(S,\phi)$ corresponding to the capped off boundary component of $S$.
\end{theorem}

Theorem \ref{thm:main} then follows immediately from Theorems \ref{otcapoff} and \ref{thm:cappingsurgery}. 

The proof of {\bf (A)} -- that our examples are universally tight -- relies upon a result of Colin and Honda from \cite{ch}. It is interesting to observe that when one proves, using their result, that a contact manifold is universally tight, one gets for free that Legendrian sugery on any link in this contact manifold is also tight ({\em cf.}\/ Theorem \ref{thm:utightfdtc}). On the other hand, some admissible transverse surgeries in our examples are overtwisted. It follows, in particular, that there exist admissible transverse surgeries which cannot be achieved via Legendrian surgeries. As alluded to above, we will prove a more precise, local version of this result in Section \ref{sec:legtrans}.

\begin{remark}\label{rmk:contacthomology}
Our examples, in combination with Colin and Honda's work, also illustrate an interesting phenomenon having to do with Eliashberg and Hofer's contact homology.  Specifically, the contact manifolds in our examples are strongly symplectically cobordant to overtwisted contact manifolds and yet have non-vanishing contact homologies. To put this in context, recall that an \emph{exact} symplectic cobordism from one contact manifold to another gives rise to a map from the contact homology of the second to that of the first. Since the contact homology of an overtwisted contact manifold vanishes \cite{yau}, our examples show that no such map exists for strong symplectic cobordisms in general. 

This is not a new observation. Gay showed in \cite{gay2} that any contact manifold with Giroux torsion greater than one is strongly symplectically cobordant to an overtwisted manifold. Many such contact manifolds (including the tight 3-tori with Giroux torsion greater than one) have contact forms with no contractible Reeb orbits, and therefore have non-vanishing contact homologies. Since Gay's work, Wendl \cite{wen2} and Latschev and Wendl \cite{latwen} have discovered further examples of contact manifolds with non-vanishing contact homologies that are strongly symplectically cobordant to overtwisted manifolds.

\end{remark}


\subsection{An Invariant of Transverse Knots and Uniform Thickness}

Suppose $K$ is a transverse knot in a tight contact manifold. We define $\slope(K)\subset \mathbb{R}\cup\{\pm\infty\}$ to be the set of slopes for which admissible transverse surgery on $K$ is tight. It is clear that $\slope$ defines an invariant of transverse knots. Using our examples and the results in Section \ref{sec:legtrans}, we prove the following rather odd fact about this invariant.

\begin{theorem}\label{prop:legadmintervals}
There exist transverse knots $K$ for which $\slope(K)$ is non-closed and disconnected. 
\end{theorem}
 
It would be interesting to compute $\slope(K)$ even for transverse knots in the tight contact structure on $S^3$. In particular, it would be interesting to determine whether $\slope$ can distinguish transverse representatives of a knot type which are not distinguished by their self-linking numbers. An invariant which can do this is said to be \emph{effective}. Historically, effective transverse invariants have been hard to come by. We leave the following question open.

\begin{question}
Is $\slope$ an effective invariant of transverse knots?
\end{question}

Our examples also lead to an interesting observation about {\em uniform thickness}, which has been an important notion in the classification of Legendrian knots. Recall from \cite{EtnyreHonda05} that a knot type $\mathcal{K}$ in a contact manifold $(M,\xi)$ is called uniformly thick if for any solid torus $S$ embedded in $M$ whose core is in the knot type $\mathcal{K}$ there is another solid torus $S'$ embedded in $M$ such that $S\subset S'$ and $S'$ is the standard neighborhood of a Legendrian representative of $\mathcal{K}$ whose Thurston-Bennequin invariant is maximal among all such representatives.  The relationship between Legendrian surgery and admissible transverse surgery is particular nice for uniformly thick knot types. 

\begin{theorem}\label{thm:utgood}
Suppose $\mathcal{K}$ is a uniformly thick knot type in the contact manifold $(M,\xi)$. Then every admissible transverse surgery on a transverse representative $K$ of $\mathcal{K}$ is equivalent to a sequence of Legendrian surgeries on a Legendrian link contained in some neighborhood of $K$. 
\end{theorem}

\begin{corollary}
\label{cor:uthickt}
Suppose $(M,\xi)$ is a tight contact manifold on which Legendrian surgery preserves tightness, and that $\mathcal{K}$ is a uniformly thick knot type in $(M,\xi)$. Then, for any transverse representative $K$ of $\mathcal{K}$, we have
\[
t(K)=[-\infty,\overline{\tb}(\mathcal{K})), 
\]
where $\overline{\tb}(\mathcal{K})$ is the maximal Thurston-Bennequin invariant among Legendrian representatives of $\mathcal{K}$. 
\end{corollary}

This corollary shows that $t$ can only hope to be an interesting (or, at least, new) invariant for transverse knots in \emph{non}-uniformly thick knot types or knot types in contact manifolds on which Legendrian surgery does not preserve tightness.

Examples of uniformly thick knot types abound; see \cite{EtnyreHonda05, LaFountain1, LaFountain2}. In contrast, it is generally difficult to find non-uniformly thick knot types, though the unknot and the positive torus knots are such knot types \cite{EtnyreHonda05}. Our examples provide a new infinite family of non-uniformly thick knots. Suppose $(S,\phi)$ is a genus one, two boundary component open book supporting one of our examples, of the sort whose construction was described in the previous subsection. According to Theorem \ref{thm:cappingsurgery} and the fact that $(S,\phi)$ satisfies the property {\bf (B)}, each of its binding components is a transverse knot on which admissible 0-surgery is overtwisted. Moreover, as discussed at the end of Subsection \ref{ssec:cappingintro}, Legendrian surgery preserves tightness for the contact manifold supported by $(S,\phi)$. Theorem \ref{thm:utgood} therefore implies the following.

\begin{corollary}\label{cor:ut}
Each binding components of $(S,\phi)$ belongs to a non-uniformly thick knot type. 
\end{corollary}

\subsection{Contact Invariants, Fillability and Deformations of Taut Foliations}
Our examples also offer new information about the contact invariant in Heegaard Floer homology.
Recall that this invariant assigns to $(M,\xi)$ a class $c(\xi)\in\hf(-M)$ which vanishes when $\xi$ is overtwisted and is non-zero when $\xi$ is strongly fillable \cite{osz1,ghiggini}. There \emph{are} tight contact structures with vanishing invariant; for instance, $c(\xi)=0$ whenever $(M,\xi)$ contains Giroux torsion \cite{ghihonvh}. Furthermore, Massot and Wendl  have separately discovered infinite families of tight, torsion-free contact manifolds with trivial invariants \cite{massot,wen3}. 
({That Wendl's examples have trivial Heegaard Floer contact invariants follows from the HF=ECH correspondence of Kutluhan-Lee-Taubes \cite{klt1,klt2,klt3} and Colin-Ghiggini-Honda \cite{cgh1,cgh2}.})
Their examples are Seifert fibered spaces over surfaces with genus at least 3 and therefore contain incompressible tori. In contrast, we prove the following result. 

\begin{theorem}\label{thm:atoroidalcneq0}
There exist infinitely many atoroidal, tight contact manifolds with trivial contact invariants.
\end{theorem}

Our examples are, in particular, hyperbolic and universally tight. They are not strongly fillable since their contact invariants vanish. Moreover, infinitely many of them are rational homology 3-spheres and are thus non-weakly fillable as well \cite{oono}. This subset of our examples gives rise to the following corollary.

\begin{corollary}\label{cor:atoroidaltight}
There exist infinitely many hyperbolic, universally tight contact manifolds that are not weakly fillable.
\end{corollary}

There are many examples in the literature of tight, non-weakly fillable contact manifolds ({\em cf.}\/ \cite{EH3,ghigginihonda,wenneid,plavhm,wen3}), but all either contain Giroux torsion or are Seifert fibered spaces. It appears that ours are the first such examples which are hyperbolic.

Recall that a non-weakly fillable contact structure cannot be the deformation of a co-orientable taut foliation \cite{et}. Therefore, Corollary \ref{cor:atoroidaltight} reproduces and improves upon a recent result of Lekili and Ozbagci which states that there exist atoroidal, universally tight contact structures that are not deformations of co-orientable taut foliations \cite{lekiliozbagci}. Their examples are Stein fillable and, hence, quite different from ours.

We end with a discussion on the relationship between open books and taut foliations. Suppose $\phi$ is a boundary-fixing diffeomorphism of $S$ that is freely isotopic to a pseudo-Anosov map, and let $B_1,\dots,B_k$ denote the boundary components of $S$. In \cite{hkm1}, Honda, Kazez and Mati{\'c} define the \emph{fractional Dehn twist coefficient} (FDTC) of $\phi$ around $B_i$ to be, roughly, the amount of twisting around $B_i$ in the free isotopy above. They prove in \cite{hkm2} that an open book with connected binding whose monodromy has FDTC at least one supports a contact structure which is the deformation of a co-orientable taut foliation. This prompted the following natural question.

\begin{question}[\rm{\cite[Question 7.3]{bald5}}]\label{ques:fdtcfoliation}
Does the analogous result hold for open books with disconnected binding?
\end{question}

Our examples show that the answer to this question is ``no."

\begin{theorem}\label{thm:fdtc}
There exist open books whose FDTCs are arbitrarily large, but whose compatible contact structures are not deformations of co-orientable taut foliations.
\end{theorem}

\subsection*{Organization} In Section \ref{sec:prelim}, we provide background on Legendrian surgery and admissible transverse surgery. In Section \ref{sec:legtrans}, we establish several connections between these two surgery operations, resulting in a precise formulation of Theorem~\ref{thm:same}. In this section, we also prove Theorem~\ref{thm:utgood} and Corollary \ref{cor:uthickt}. In Section~\ref{sec:examples}, we describe our examples, and show that they are universally tight using Colin and Honda's work. In Subsection \ref{ssec:hfmap}, we discuss the map on Heegaard Floer homology associated to the operation of capping off, and use this to prove Theorems \ref{otcapoff}, ~\ref{thm:atoroidalcneq0} and~\ref{thm:fdtc} and Corollary \ref{cor:atoroidaltight}. In Subsection \ref{ssec:admissiblecap}, we describe the relationship between capping off and admissible transverse surgery, and prove Theorems~\ref{thm:main}, \ref{thm:cappingsurgery} and \ref{prop:legadmintervals}. 

\subsection*{Acknowledgements} The authors thank Vincent Colin, Tobias Ekholm, Yasha Eliashberg, Ko Honda, Janko Latschev and Chris Wendl for helpful correspondence. JAB was partially supported by NSF grant DMS-1104688 and JBE was partially supported by NSF grant DMS-0804820 and thanks the University of Texas, Austin for its hospitality while working on parts of this paper.

\section{Preliminaries}
\label{sec:prelim}

In this section, we recall the definitions of admissible transverse surgery and Legendrian surgery. We will assume throughout this section that the reader is familiar with convex surface theory; see \cite{EH2, honda2} for the necessary background.

Given a torus $T$ and an basis $(\lambda,\mu)$ for $H_1(T;\Z)$, every homologically essential simple closed curve $\gamma\subset T$ is homologous to $q\lambda + p \mu$, for some $p$ and $q$ which are relatively prime. As an unoriented curve, $\gamma$ is therefore determined by the rational number $p/q$, which we call the \emph{slope} of $\gamma$. When $T$ is the boundary of a solid torus, we will choose $\mu$ to be a meridian and $\lambda$ to be a preferred longitude, and we orient $\mu$ and $\lambda$ so that $\mu \cap \lambda = +1$.

\subsection{Admissible Transverse Surgery and Contact Cuts}
\label{ssec:adm}

Consider the open solid torus $U=\R^2\times S^1$ with the contact structure $\xi_0=\ker(\cos f(r)\, d\phi+f(r)\sin f(r) \, d\theta)$ where $f:[0,\infty)\to [0,\pi)$ is an increasing surjective function of $r$. The contact manifold $(U,\xi_0)$ is covered by the standard tight contact structure on $\R^3$, and is therefore tight.  For each $a>0$, we orient the torus $\{(r,\theta, \phi): r=a\}$ as the boundary of the solid torus $\{(r,\theta, \phi): r\leq a\}$, and we let $(\lambda,\mu)$ be the homology basis for this torus given (as unoriented curves) by $\mu=\{\phi = {\rm constant}\}$ and $\lambda=\{\theta = {\rm constant}\}$. 

The characteristic foliations of $\xi$ on the tori $\{(r,\theta, \phi): r=a\}$ are linear with slopes that increase monotonically from $-\infty$ to $\infty$ as $a$ ranges from $0$ to $\infty$. Given $s\in \R$, we let $T_s$ denote the torus $\{(r,\theta, \phi): r=a\}$ whose characteristic foliation has slope $s$, and we let $S_s$ denote the solid torus in $U$ bounded by $T_s$. For $s'>s$, we denote by $S_{s,s'}$ the thickened torus $\overline{S_{s'}-S_s}$. Note that any neighborhood of $T_s$ contains $S_{s-\delta,s+\delta}$ for some $\delta$. The solid tori $S_s$ and the thickened tori $S_{s,s'}$ inherit contact structures from $\xi_0$; accordingly, $S_s$ and $S_{s,s'}$ will refer to smooth manifolds or contact manifolds depending on context. 

If $T$ is a torus in some contact manifold $(M,\xi)$ and $\phi:T\to T_s$ is a diffeomorphism which sends the characteristic foliation on $T$ to that on $T_s$, then $\phi$ extends to a contactomorphism from a neighborhood of $T$ to a neighborhood of $T_s$. It follows that if $T$ is a torus in $(M,\xi)$ with linear characteristic foliation, then some neighborhood $T\times[-1,1]$ of $T=T\times\{0\}$ is contactomorphic to $S_{a,b}$ for some $a<b$. 

Let $K$ be a transverse knot in a contact manifold $(M,\xi)$. It is well-known ({\em cf.}\/ \cite{EH2}), that $K$ has a neighborhood $N$ which is contactomorphic to $S_s$ for some $s>0$, via a map which identifies $K$ with the core $\{r=0\}$ of $S_s$. We refer to such an $N$ as \emph{a standard neighborhood} of $K$. If $N$ has a preferred longitude, we choose our contactomorphism from $N$ to $S_s$ so that it sends this longitude to $\lambda$ (note that as you change which preferred longitude is sent to $\lambda$, the value of $s$ will, in general, change). In a slight abuse of notation, we will often just equate $N$ with $S_s$. For any rational number $r\in(-\infty, s)$,  there is a natural contact structure $\xi_{K}(r)$ on the manifold $M_K(r)$ obtained from $M$ by performing $r$-surgery on $K$. We say $(M_K(r),\xi_K(r))$ is obtained from $(M,\xi)$ by \dfn{admissible transverse $r$-surgery on $K$}. Below, we describe two constructions of $\xi_K(r)$. The first construction is essentially the original definition due to Gay in \cite{gay}, but expressed in the notation developed above. 

For the first construction, notice that there is a solid torus $S_{r}$ in $N=S_s$ whose boundary has characteristic foliation of slope $r$. We remove a slightly larger solid torus $S_{b}$ from $(M,\xi)$, where $b\in(r,s)$. Let $S$ denote the solid torus that is glued to $\overline{M - S_{b}}$ to form $M_K(r)$. The restriction of $\xi$ to $\overline{M - S_{b}}$ induces a contact structure $\xi'$ on $\overline{M_K(r)- S}\cong \overline{M - S_{b}}$, and there is a unique way of extending $\xi'$ to a contact structure $\xi_K(r)$ on all of $M_K(r)$ so that the restriction of $\xi_K(r)$ to $S$ is universally tight. Specifically, there is a diffeomorphism from $\partial S$ to some $T_c$ sending the characteristic foliation of $\xi'$ on $\partial S$ to the characteristic foliation on $T_c$ (this $c$ is uniquely determined by a preferred longitude on $\partial S$ by requiring that the diffeomorphism send this longitude to $\lambda$). We then extend $\xi'$ over $S$ using the contact structure on $S_c$. Notice that the core of $S$ is a transverse knot $K'$, and that $S$ is a standard neighborhood of $K'$. We call $K'$ the \dfn{dual transverse curve to $K$}. One may easily check that $M_K(r)-K'$ is contactomorphic to $M-S_{r}$. 

For the second construction of $\xi_K(r)$, we recall the notion of a \emph{contact cut} \cite{Lerman01}. Suppose $Y$ is a 3-manifold with a torus boundary component $T$, and suppose that there is a free (proper) $S^1$-action on $T$. Let $Y'$ be the quotient space obtained from $Y$ by identifying  points of $T$ in the same orbit, and let $K'\subset Y'$ be the set of points with more than one pre-image under the quotient map from $Y$ to $Y'$. We claim that $Y'$ has a natural smooth structure. This only needs to be checked at points on $K'$. To that end, let $N=T\times [0,1]$ be a collar neighborhood of $T$ in $Y$, and let $f:N\rightarrow [0,1]$ be the obvious projection map. We can assume that the $S^1$-action on $T = T\times\{0\}$ extends to an $S^1$-action on all of $N$ under which $T\times\{t\}$ is invariant for every $t\in[0,1]$. Thinking of $S^1$ as the unit circle in $\C$, we extend this $S^1$-action on $N$ to an $S^1$-action on $N\times \C$ by $\theta\cdot(p,z)=(\theta\cdot p, \theta^{-1} z)$. Now, let $F:N\times \C\to \R$ be the map given by $F(p,z)= f(p)-|z|^2$. 

One may easily check that $0$ is a regular value of $F$ and that the $S^1$-action on $N\times \C$ restricts to a free $S^1$-action on $F^{-1}(0)$. Therefore, $F^{-1}(0)/S^1$ is a smooth 3-manifold; moreover, it is clearly homeomorphic to the quotient space $(T\times[0,1])'\subset Y'$ obtained from $T\times[0,1]$ by collapsing the orbits of the $S^1$-action on $T=T\times\{0\}$ to points. As $(T\times[0,1])'$ is a neighborhood of $K'$, we have verified that $Y'$ is a smooth 3-manifold. Notice that the orbits of the $S^1$-action on $T$ describe closed curves of some slope $r$ (once we have chosen a longitude on $T$). It is easy to see that $Y'$ is homeomorphic to the $r$-Dehn filling of $Y$ -- that is, the manifold obtained by gluing a solid torus $S$ to $Y$ so that the meridian of $\partial S$ is glued to a curve on $T$ of slope $r$. 

Now, suppose that $\xi$ is a contact structure on $Y$, defined near $T$ as the kernel of some contact form $\alpha$. Suppose further that the orbits of the $S^1$-action on $T$ are leaves of the characteristic foliation of $\xi$.  This characteristic foliation is therefore linear, and we can assume that the collar neighborhood $N=T\times [0,1]$ is contactomorphic to $S_{a,b}$ for some $a<b$. The $S^1$-action on $T = T\times \{0\}$ thus extends to an $S^1$-action on $N$ whose orbits on $T\times\{t\}$ are leaves of the characteristic foliation on $T\times\{t\}$. In particular, we may assume that $\alpha$ is invariant under the $S^1$-action on $N$. Now, consider the contact form $\beta=\alpha + x\, dy -y\, dx$ on $N\times \C$, where $(x,y)$ are the coordinates on $\C$. This form is clearly invariant under the $S^1$-action on $N\times \C$, and descends to a contact form $\alpha'$ on $(T\times[0,1])'=F^{-1}(0)/S^1$ ({\em cf.}\/ \cite{Geiges97}). It is easy to check that the contact structure $\xi'=\ker\alpha'$ agrees with the restriction of $\xi$ to $Y'-K'\cong Y-T$, and that $K'$ is a transverse knot in $(Y',\xi')$. The contact manifold $(Y',\xi')$ is said to be the result of a \emph{contact cut} along $T$; see \cite{Lerman01}. 

The construction of $\xi_K(r)$ via contact cuts proceeds as follows. Suppose that $K$ is a transverse knot in $(M,\xi)$ with a standard neighborhood $N=S_s$. Choose any rational number $r\in (-\infty, s)$, and consider the smaller neighborhood $S_{r}$ of $K$ whose boundary has characteristic foliation of slope $r$. There is an obvious free $S^1$-action on $\partial S_{r}$ whose orbits are leaves of the characteristic foliation.  Note that $\xi$ restricted to $Y=\overline{M-S_{r}}$ with this $S^1$-action on $\partial (\overline{M-S_{r}})$ is a contact structure as in the previous paragraph. Performing a contact cut along $\partial (\overline{M-S_{r}})$ results in a contact manifold $(Y',\xi')$ which is easily identified with $(M_K(r), \xi_K(r))$. Note that the knot $K'$ defined in the paragraph above is also the dual transverse curve to $K$; in particular, the restriction of $\xi'$ to $Y'-K'$ is contactomorphic to the restriction of $\xi$ to $M-S_{r}$.  

A simple argument yields the following lemma. 

\begin{lemma}
\label{lem:admutight}
For any $r<s$, admissible transverse $r$-surgery on the core $K$ of $S_s$ results in a solid torus $S$ which is a standard neighborhood of the dual transverse curve $K'$. In particular, $S$ is universally tight.
\qed
\end{lemma}

\subsection{Surgery on Legendrian Knots}\label{ssec:csurgery}
\label{ssec:leg}

Let $\xi_1$ be the contact structure on $\R^2\times S^1$ defined by $\xi_1=\ker(dy-x\, d\theta)$.  The contact manifold $(\R^2\times S^1,\xi_1)$ is covered by the standard tight contact structure on $\R^3$, and is therefore tight. For $a>0$, let $\lS_a$ denote the solid torus $\{(x,y,\theta): x^2+y^2\leq a\}$. It is a neighborhood of the Legendrian curve $C=\{(x,y,\theta): x=y=0\}$, and its boundary $\partial \lS_a$ is convex with two dividing curves parallel to $\{(x,y,\theta): x=0, y=a\}$. All of the solid tori $\lS_a$ are, more or less, contactomorphic according to the following basic result of Kanda.

\begin{theorem}[Kanda 1997, \cite{Kanda97}]\label{Kanda}
Let $\Gamma$ be a pair of longitudinal curves on the boundary of a solid torus $S$. Let $\mathcal{F}$ be a singular foliation on $\partial S$ that is divided by $\Gamma.$ Then there is a unique (up to isotopy) tight contact structure on $S$ whose characteristic foliation on $\partial S$ is $\mathcal{F}$. \qed
\end{theorem}

Suppose $L$ is a Legendrian knot in a contact manifold $(M,\xi)$. It is well-known that $L$ has a neighborhood $N$ which is contactomorphic to $\lS_a$ for some (hence, all) $a>0$. We refer to such an $N$ as \emph{a standard neighborhood} of $L$. Let $(\lambda,\mu)$ be the homology basis for $\partial \lS_a$ given by $\mu = \{\theta = {\rm constant}\}$ and $\lambda = \{x={\rm constant}, \, y={\rm constant}\}$. Note that $\lambda$ is the longitude specified by the contact framing of $C$, and is parallel to the dividing curves. The contactomorphism between $N$ and $\lS_a$ therefore identifies $\lambda$ with the longitude on $\partial N$ given by the contact framing of $L$; this will serve as the preferred longitude on $\partial N = -\partial (\overline{M-N})$.

Suppose $L$ is a Legendrian knot in $(M,\xi)$ and let $N$ be a standard neighborhood of $L$.   Let $\lS = D^2 \times S^1$ be another solid torus with preferred longitude given by $\{pt\}\times S^1$. We perform $\pm 1$-surgery on $L$, with respect to its contact framing, by removing $N$ from $M$ and gluing in the solid torus $\lS$ according to the map $\phi: \partial \lS\to -\partial (\overline{M- N})$ determined by the matrix
\[
\begin{pmatrix}
1 &1\\
0& \pm 1
\end{pmatrix},
\]
with respect to the preferred longitude-meridian coordinate systems on $\partial \lS$ and $-\partial (\overline{M- N})$. We denote this surgered manifold by $M_L(\langle \pm 1\rangle)$. 
The contact structure $\xi$ restricts to a contact structure $\xi'$ on $\overline{M_L(\langle \pm 1\rangle)-\lS}$, and, by Theorem \ref{Kanda}, there is a unique way to extend $\xi'$ over $\lS$ so that it is tight on $\lS$. We denote the resulting contact structure by $\xi_L(\langle \pm 1\rangle)$, and say that $(M_L(\langle \pm 1\rangle),\xi_L(\langle \pm 1\rangle))$ is the result of \emph{contact $\pm 1$-surgery} on $L$. Contact $-1$-surgery on $L$ is also commonly called \emph{Legendrian surgery} on $L$. 

For $b\in[-a,a]$, the knot $L_b\subset \lS_a$ given by $L_b=\{(x,y,\theta): x=0, y=b\}$ is Legendrian and is Legendrian isotopic to the core $C$ of $\lS_a$. Therefore, the contactomorphism from $\lS_a$ to a standard neighborhood $N$ of $L\subset (M,\xi)$ sends $L_b$ to a Legendrian knot which is Legendrian isotopic to $L$. We call this knot a \emph{Legendrian push off} of $L$.
\begin{lemma}[Ding and Geiges 2004, \cite{DG2}]\label{DG}
Suppose $L$ is a Legendrian knot in $(M,\xi)$, and let $L'$ be a Legendrian push off of $L$. After performing contact $\pm 1$-surgery on $L$ and contact $\mp 1$-surgery on $L'$, we obtain a contact structure on $M$ that is isotopic to $\xi$.\qed
\end{lemma}

\section{Connections between Legendrian and admissible transverse surgeries}
\label{sec:legtrans}
In this section, we establish several results concerning the relationships between admissible transverse surgery and Legendrian surgery. Some of these have implications for our invariant $t(K)$; others will be used to prove results related to uniform thickness. Our main theorems are stated below. The first is a generalization of a result of Gay from \cite{gay} that says that \emph{every} Legendrian surgery can be achieved via an admissible transverse surgery. 

\begin{theorem}\label{thm:LisT}
Let $L$ be a Legendrian knot in some contact manifold. Suppose $N$ is a standard neighborhood of $L$. Then, the contact manifold obtained via Legendrian surgery on $L$ can also be obtained via an admissible transverse surgery on a transversal push off $K$ of $L$. Moreover, we can arrange that the neighborhood of $K$ used in defining this transverse surgery is entirely contained in $N$. 
\end{theorem}

The situation is more complicated when considering which admissible transverse surgeries are ``equivalent" to Legendrian surgeries. 

\begin{theorem}\label{SimpleSurg}
Let $K$ be a transverse knot in some contact manifold. Suppose $N$ is a standard neighborhood of $K$ such that the characteristic foliation on $\partial N$ is linear with slope $a$, where $n<a<n+1$ for some integer $n$. 
Then, for any rational number $s<n$, admissible transverse $s$-surgery on $K$ can also be achieved by Legendrian surgery on some Legendrian link in $N$.
\end{theorem}

Recall from the Introduction that $t(K)$ is the set of slopes for which admissible transverse surgery on $K$ is tight. The following is an immediate corollary of Theorem \ref{SimpleSurg}.

\begin{corollary}
Let $K$ be a transverse knot in a tight contact manifold $(M,\xi)$. If Legendrian surgery on any Legendrian link in $(M,\xi)$ is also tight (for example, if $(M,\xi)$ is fillable or has non-zero Heegaard Floer contact invariant), then $t(K)$ contains the interval $[-\infty, n)$. \qed
\end{corollary}

The theorem below shows that some admissible transverse $s$-surgeries for $s >n$ can also be achieved via Legendrian surgery. 

\begin{theorem}\label{thm:extraslopes}
Let $K$ be a transverse knot in some contact manifold. Suppose $N$ is a standard neighborhood of $K$ such that the characteristic foliation on $\partial N$ is linear with slope $a$, where $n<a<n+1$ for some integer $n$. Then, there is a decreasing sequence $\{d_k\}$ of numbers less than $a$ and converging to $n$ such that, for all $k$, admissible transverse $d_k$-surgery on $K$ can also be achieved by Legendrian surgery on some Legendrian knot in $N$. 
\end{theorem}

The following is an immediate corollary of Theorem \ref{thm:extraslopes}.

\begin{corollary}
Let $K$ be a transverse knot in a tight contact manifold $(M,\xi)$. If Legendrian surgery on any Legendrian link in $(M,\xi)$ is also tight, then $t(K)$ contains an infinite number of slopes in the interval $(n,n+1)$. 
\end{corollary}

As indicated in the theorem below, it is not always true that admissible transverse surgery can also be achieved (locally) via Legendrian surgeries. 

\begin{theorem}\label{excludedsurgery}
Let $K$ be a transverse knot in some contact manifold $(M,\xi)$.
Suppose $N$ is a standard neighborhood of $K$ such that the characteristic foliation on $\partial N$ is linear with slope $a$, where $n<a<n+1$ for some integer $n$. Then, there is a descending sequence $\{d_k\}$ of numbers less than $a$ and converging to $n$ such that, for all $k$, admissible transverse $d_k$-surgery on $K$ is not the result of Legendrian surgery on any Legendrian link in $N$. Moreover, there is an ascending sequence $\{a_k\}$ of numbers less than $a$ and converging to $n+1$ with exactly the same property. 
\end{theorem}

\begin{remark}
Recall from the discussion in Subsection \ref{ssec:cappingintro} that the examples used in the proof of our main theorem will show that there are admissible transverse surgeries that cannot be achieved by \emph{any} sequence of Legendrian surgeries, whether these Legendrian surgeries take place near the original transverse knot, as in Theorem \ref{excludedsurgery}, or otherwise.
\end{remark}

We prove the above theorems in the next two subsections. In the last subsection, we discuss the relationship between Legendrian and admissible transverse surgery for uniformly thick knot types, and we prove Theorem \ref{thm:utgood} and Corollary \ref{cor:uthickt}.

\subsection{Admissible Transverse Surgeries Which Are Local Legendrian Surgeries}
\label{ssec:admleg}

In this subsection we identify admissible transverse surgeries that can be achieved locally by Legendrian surgeries (Theorems~\ref{SimpleSurg} and~\ref{thm:extraslopes}), and we show that every Legendrian surgery can be achieved by an admissible transverse surgery (Theorem~\ref{thm:LisT}). 

Recall from Subsection~\ref{ssec:adm} that, for an interval $(r,s)$, $S_{r,s}$ denotes the thickened torus $T^2\times[0,1]$ with the contact structure which is tangent to the $[0,1]$-factor and rotates from slope $r$ on $T^2\times\{0\}$ to slope $s$ on $T^2\times\{1\}$, hitting every slope in $(r,s)$ exactly once. This contact structure is said to be \emph{minimally twisting}. Extending the notation of Subsection~\ref{ssec:adm}, we allow for $r,s = \pm \infty$, and we think of $(r,s)$ as an interval on the circle obtained from the real line by identifying $+\infty$ with $-\infty$; in particular, if $r>s$, then $(r,s)$ is the union $(r,\infty]\cup[-\infty,s)$.

With this notation, observe that $S_a$ is obtained from $S_{-\infty, a}$ by performing a contact cut along $T^2\times\{0\}$. Moreover, admissible transverse $s$-surgery on the core of $S_a$ is obtained from $S_{s,a}$ by performing a contact cut along $T^2\times\{0\}$, as described in Subsection~\ref{ssec:adm}. Thus, to prove Theorem~\ref{SimpleSurg}, we need only show that $S_{s,a}$ can be obtained from $S_{-\infty, a}$ by a sequence of Legendrian surgeries (as in Subsection \ref{ssec:adm}, we identify the standard neighborhood $N$ in Theorem~\ref{SimpleSurg} with $S_a$ via a contactomorphism which identifies $K$ with the core of $S_a$).  Or equivalently (by Lemma~\ref{DG}), that $S_{-\infty, a}$ can be obtained from $S_{s,a}$ from a sequence of $+1-$contact surgeries. We will prove the latter. To do so, we first determine the effect of contact $\pm 1-$surgeries on knots in the thickened tori $S_{r,s}$.

\begin{lemma}\label{surgeryandslopes}
For $r/s\in (p_0/q_0, p_1/q_1)$, recall that $T_{r/s}$ denotes the torus in $S_{p_0/q_0, p_1/q_1}$ that is linearly foliated by curves of slope $r/s$. Let $L$ be one of these curves. The result of contact $\pm 1$-surgery on $L$ is $S_{p'_0/q'_0, p_1/q_1}$, where 
\[
p'_0/q'_0=\frac{\mp r^2q_0+(1\pm rs) p_0}{(1\mp rs) q_0\pm s^2p_0}.
\]
\end{lemma}

\begin{remark} Let $L$ be as in Lemma~\ref{surgeryandslopes}. While not immediately obvious, it is nonetheless easy to check that contact $-1$-surgery on $L$ yields a thickened torus with less twisting; that is, $(p_0'/q_0',p_1/q_1)\subsetneq (p_0/q_0,p_1/q_1)$. Likewise, contact $+1$-surgery on $L$ yields a thickened torus with more twisting; that is, $(p_0/q_0,p_1/q_1)\subsetneq (p_0'/q_0',p_1/q_1)$.
\end{remark}

\begin{remark}\label{geom}

Our proof of Lemma \ref{surgeryandslopes} involves understanding how a curve on the torus of some slope changes under a positive or negative Dehn twist about a curve of some other slope. Although we will not use it explicitly in the proof, the Farey tessellation provides a nice, geometric way of understanding these kinds of changes. Recall that in the Farey tessellation of the Poincar{\'e} disk $\mathbb{D}$, rational slopes on the torus are labeled by points on $\partial \mathbb{D}$, and two slopes on the boundary are connected by a geodesic edge in $\mathbb{D}$ if and ony if the curves on the torus with these two slopes form an integral basis for $H_1(T^2;\mathbb{Z})$ (\emph{cf.} \cite{honda2} for the precise conventions). Given a slope $r/s$, we can order the slopes connected by an edge to $r/s$ by $s_k$, for $k\in \Z$, such that $s_k$ is counterclockwise of $s_{k-1}$ on $\partial \mathbb{D}$. Then a positive (resp. negative) Dehn twist about a curve of slope $r/s$ sends a curve of slope $s_k$ to one of slope $s_{k+1}$ (resp. $s_{k-1}$). From this, one can determine what such Dehn twists do to curves of any other slope by writing the slope as a ``Farey sum" of some number of copies of $s_k$ and $s_{k+1}$ for the appropriate $k$. 
\end{remark}

\begin{proof}[Proof of Lemma \ref{surgeryandslopes}]

Note that the contact framing on $L$ agrees with the framing on $L$ coming from its embedding in $T_{r/s}$. Therefore, contact $\pm1$-surgery on $L$ is topologically the same as $\pm1$-surgery with respect to the framing induced by $T_{r/s}$. Recall that $\pm1$-surgery with respect to the latter framing yields the same (topological) manifold as that obtained by cutting $S_{p_0/q_0, p_1/q_1}$ along $T_{r/s}$ into the pieces $S_{p_0/q_0, r/s} \cup S_{r/s, p_1/q_1}$, and then re-gluing along $T_{r/s}$ by $D_L^{\mp 1}$, where $D_L$ is a positive Dehn twist along $L \subset T_{r/s}$. The result of contact $\pm 1$-surgery on $L$ is therefore another thickened torus. 

Recall that $L$ is a curve in $T_{r/s}$ of slope $r/s$. It is not hard to see that the map $D^{\pm}_L$ is described by the matrix 
\[
\begin{pmatrix}
1\pm rs & \mp s^2\\
\pm r^2 & 1\mp rs
\end{pmatrix},
\]
with respect to the longitude-meridian coordinates on $T_{r/s}$. This map sends a curve of slope $p_0/q_0$ to one of slope $p'_0/q'_0=\frac{\pm r^2q_0+(1\mp rs) p_0}{(1\pm rs) q_0\mp s^2p_0}$. Therefore, contact $\pm 1$-surgery on $L$ results in a contact structure on $T^2\times[0,1]$ for which the characteristic foliation on $T^2\times\{0\}$ has slope $p'_0/q'_0$ and the characteristic foliation on $T^2\times \{1\}$ has slope $p_1/q_1$. To show that this contact structure agrees with that on $S_{p'_0/q'_0, p_1/q_1}$, we need only show that it is universally tight (up to an orientation on the contact planes, there is a unique minimally twisting, universally tight contact structure on a thickened torus \cite{honda2}). We will prove universal tightness for the contact $-1$-surgery; the proof for the contact $+1$-surgery is virtually identical. It can also be deduced from the result for contact $-1$-surgery together with Lemma~\ref{DG}. 

Let $S_{\overline{a},b}$ be the thickened torus $T^2\times [0,1]$ with universally tight contact structure such that $T^2\times\{0\}$ is convex with two dividing curves of slope $a$ and $T^2\times \{1\}$ has linear characteristic foliation of slope $b$. Note that $S_{\overline{a},b}$ is obtained from $S_{a,b}$ by perturbing the leftmost boundary component to be convex (of course, this perturbation takes place in a slightly larger thickened torus). We define $S_{a,\overline{b}}$ and $S_{\overline{a},\overline{b}}$ analogously. With this notation in place, we begin by splitting $S_{p_0/q_0, p_1/q_1}$ into three pieces. Let $r_0/s_0$ be the rational number in $(p_0/q_0, p_1/q_1)$ closest to $p_0/q_0$ for which curves of slopes $r_0/s_0$ and $r/s$ form an integral basis for $H_1(T^2;\Z)$, and let $r_1/s_1$ be the rational number in $(p_0/q_0, p_1/q_1)$ closest to $p_1/q_1$ with the same property. We then split $S_{p_0/q_0, p_1/q_1}$ as $S_{p_0/q_0, \overline{r_0/s_0}}\cup S_{\overline{r_0/s_0}, \overline{r_1/s_1}}\cup S_{\overline{r_1/s_1}, p_1/q_1}$. Note that $L\subset T_{r/s}$ is in the $S_{\overline{r_0/s_0}, \overline{r_1/s_1}}$ piece.

Since the curves of slopes $r_0/s_0$ and $r/s$ form an oriented integral basis for $H_1(-T^2,\Z)$, the Dehn twist $D_L$ sends a curve of slope $r_0/s_0$ to a curve of slope $r_2/s_2 = \frac{r_0+r}{s_0+s}$. We claim that the result of contact $-1-$contact surgery on $L\subset S_{\overline{r_0/s_0}, \overline{r_1/s_1}}$ is $S_{\overline{r_2/s_2}, \overline{r_1/s_1}}$. This claim then implies that contact $-1$-surgery on $L\subset S_{p_0/q_0, p_1/q_1}$ is the union of three universally tight pieces, $S_{p_0/q_0, \overline{r_0/s_0}}\cup S_{\overline{r_2/s_2}, \overline{r_1/s_1}}\cup S_{\overline{r_1/s_1}, p_1/q_1}$, where the first piece is glued to the second via the Dehn twist $D_L$. It follows that this union is itself universally tight (this may easily be checked as we have and explicit description of the resulting contact structure, or one may simply refer to the classification of tight structures on thickened tori in \cite{Giroux00, honda2}), completing the proof of Lemma \ref{surgeryandslopes}. All that remains is that we prove this claim.

To prove the above claim, let us think about $S_{\overline{r_0/s_0}, \overline{r_1/s_1}}$ as $T^2\times[0,1]$. We can arrange that the ruling curves on the boundary components $T^2\times\{0\}$ and $T^2\times \{1\}$ have slope $r/s$. Let $A$ be a properly embedded, convex annulus whose boundary consists of one ruling curve on $T^2\times\{0\}$ and one on $T^2\times \{1\}$. Let $N$ be an $I$-invariant neighborhood of the union $T^2\times\{0\}\cup T^2\times\{1\} \cup A$, with corners rounded. It follows from  our choice of $r_0/s_0$ and $r_1/s_1$ that the complement $\overline{S_{\overline{r_0/s_0}, \overline{r_1/s_1}}-N}$ is a standard neighborhood of $L$. 
Similarly, let $A'$ be a properly embedded, convex annulus in $S_{\overline{r_2/s_2}, \overline{r_1/s_1}}$ whose boundary consists of one ruling curve of slope $r/s$ on each of the boundary components of $S_{\overline{r_2/s_2}, \overline{r_1/s_1}}$, and let $N'$ be the corresponding $I$-invariant neighborhood. The complement of $N'$ in $S_{\overline{r_2/s_2}, \overline{r_1/s_1}}$ is also a standard neighborhood of a Legendrian divide $L'\subset T_{r/s}\subset S_{\overline{r_2/s_2}, \overline{r_1/s_1}}$ of slope $r/s$. There is clearly a contactomorphism from $N'$ to $N$, and it is easy to see that this contactomorphism sends the meridian of $L'$ to a curve of slope $-1$ on the boundary of a standard neighborhood of $L$. It follows that $S_{\overline{r_2/s_2}, \overline{r_1/s_1}}$ is obtained from $S_{\overline{r_0/s_0}, \overline{r_1/s_1}}$ by removing a standard neighborhood of $L$ and re-gluing it as prescribed by contact $-1$-surgery. 
\end{proof}

We now return to the proof of Theorem \ref{SimpleSurg}. 

\begin{proof}[Proof of Theorem \ref{SimpleSurg}]
Suppose $n<a<n+1$ and $s<n$ as in the hypothesis of Theorem \ref{SimpleSurg}. Recall from the beginning of this section that we need only show that $S_{-\infty,a}$ can be obtained from $S_{s,a}$ via a sequence of contact $+1-$surgeries. To simplify the notation, let us assume that $n=-1$ (we can do this since there is a diffeomorphism of $T^2\times[0,1]$ which preserves the meridian and sends any longitude to any other longitude). 

Then $s<-1$, and there is a negative integer $-m$ such that $-m-1\leq s < -m$. If $s=-m-1$, then a negative Dehn twist along a curve of slope $-m$ sends a curve of slope $s$ to one of slope $-\infty$. It then follows from Lemma \ref{surgeryandslopes} and its proof that $S_{-\infty,a}$ is the result of contact $+1$-surgery on $L\subset S_{s,a}$, where $L$ is a leaf of the characteristic foliation on $T_{-m}\subset S_{s,a}$, and we are done. Below, we see that the more general case can be reduced to this one.

Let $r_k=\frac{k(-m)+(-m-1)}{k+1}$. For integers $k\geq 0$, these are rational numbers in the interval $[-m-1,-m]$ which start at $r_0=-m-1$, increase monotonically with $k$, and approach $-m$ as $k$ approaches $\infty$. These numbers are constructed precisely so that a curve of slope $r_k$ and a curve of slope $-m$ form an oriented integral basis for $H_1(-T^2,\Z)$. Let $k\geq 0$ be such that $s\in[r_k,r_{k+1}]$. Then we can write \[s = \frac{a[k(-m)+(-m-1)] + b[(k+1)(-m)+(-m-1)]}{a(k+1) + b(k+2)}\] for some non-negative integers $a,b$. Since a negative Dehn twist along a curve of slope $-m$ sends a curve of slope $r_{j+1}$ to one of slope $r_{j}$ for all $j\in \Z$, the same Dehn twist sends a curve of slope $s$ to one of slope \[\frac{a[(k-1)(-m)+(-m-1)] + b[k(-m)+(-m-1)]}{ak + b(k+1)}.\] A composition of $k+1$ such Dehn twists therefore sends a curve of slope $s$ to one of slope \begin{equation}\label{eqn:s''1}s'=\frac{b(-m-1)-a} {b}.\end{equation}

If $b=0$, then $s' = -\infty$, and we are done. If $a=0$, then $s' = -m-1$ and we are done by the discussion in the previous paragraph. Suppose, now, that $a$ and $b$ are both positive. Write $a = lb+r$, where $l\geq 0$ and $0\leq r<b$. If $r=0$, then $s' = -m-l-1$, and a negative Dehn twist along a curve of slope $-m-l$ sends a curve of slope $s'$ to one of slope $-\infty$, and we are done. Let us then assume that $r>0$, and write \begin{equation}\label{eqn:s''}s' = \frac{b(-m-l-1)-r}{b}.\end{equation} Since $r<b$, $s'$ is in the interval $[-m-l-2,-m-l-1],$ and we apply the argument from the preceding paragraph over again, substituting $m+l+1$ for $m$ everywhere. In particular, we start by writing \begin{equation}\label{eqn:s''2}s' = \frac{a'[k'(-m-l-1)+(-m-l-2)] + b'[(k'+1)(-m-l-1)+(-m-l-2)]}{a'(k'+1) + b'(k'+2)}\end{equation} for some integers $a', b', k' \geq 0$. As above, a composition of $k'+1$ negative Dehn twists along a curve of slope $-m-l-1$ sends a curve of slope $s'$ to one of slope \begin{equation}\label{eqn:s'''}s''=\frac{b'(-m-l-2)-a'} {b'}.\end{equation}

Comparing the expressions in (\ref{eqn:s''}) and (\ref{eqn:s''2}), we find that $r = a'+b'$. If either $a'$ or $b'$ is zero, then we are done, as above. Otherwise, $a'<r\leq a$ and $b'<r<b$. In other words, we have obtained in (\ref{eqn:s'''}) an expression of the same form as that in (\ref{eqn:s''1}), but where we have replaced $a$ and $b$ by strictly smaller positive integers and $m$ by a smaller negative integer. We can then repeat the argument in the preceding paragraph, applying negative Dehn twists, until we obtain a curve of slope \[s^{(i+1)} = \frac{b^{(i)}(-m-t) - a^{(i)}}{b^{(i)}},\] where $t>0$ and either $b^{(i)}$ or $a^{(i)}$ is zero, in which case we are done.
\end{proof}

The proof of Theorem \ref{thm:extraslopes} is considerably less involved.

\begin{proof}[Proof of Theorem \ref{thm:extraslopes}]
Suppose $n<a<n+1$ as in the hypothesis of Theorem \ref{thm:extraslopes}. To simplify notation, let us assume that $n=0$. Let $k$ be a positive integer such that $1/k<a$. Let $L\subset T_{1/k}$ be a foliation curve. By the formula in Lemma \ref{surgeryandslopes}, contact $+1$-surgery on $L$ sends the curve of slope $d_k=\frac{k-1}{(k-1)(k+1)+1}$ to a curve of slope $-\infty$. It follows, as in the proof of Theorem \ref{SimpleSurg}, that admissible transverse $d_k$-surgery on $K$ is the same as Legendrian surgery on $L$. 
\end{proof}

\begin{remark}
A more careful analysis allows one to find whole intervals of slopes between $n$ and $a$ for which admissible transverse surgery can also be achieved via Legendrian surgery on some link in $N$. For example, one can show that if $a\in (1/2,1)$, then admissible transverse $s$-surgery for any $s \in [2/9,1/3)$ can be achieved via Legendrian surgery. 
\end{remark}

We conclude this subsection with the proof of Theorem~\ref{thm:LisT}, which states that any Legendrian surgery can also be achieved locally via an admissible transverse surgery.

\begin{proof}[Proof of Theorem \ref{thm:LisT}]
Let $L$ be a Legendrian knot in $(M,\xi)$ and $N$ a standard neighborhood of $L$, as in the hypothesis of Theorem \ref{thm:LisT}. It suffices to work with a local model for $N$. To this end, let $S_{\overline{0}}$ denote the contact solid torus $S_{-1}\cup S_{-1,\overline{0}}$. Since $S_{\overline{0}}$ is a contact submanifold of  $S_{\epsilon}$ for some $\epsilon>0$, it is tight. Moreover, $\partial S_{\overline{0}}$ is convex with dividing set consisting of two curves of slope $0$. Theorem \ref{Kanda} then implies that $S_{\overline{0}}$ is contactomorphic to $N$. In a slight abuse of notation, we will therefore simply equate $N$ with $S_{\overline{0}}$.

Note that the core $K\subset S_{-1}$ is a transversal pushoff of $L$. To prove Theorem \ref{thm:LisT}, it therefore suffices to show that Legendrian surgery on $L$ has the same effect on $S_{\overline{0}}$ as does admissible transverse $-1$-surgery on $K$. The latter surgery is obtained by removing $S_{-1}$ from $S_{\overline{0}}$ and then performing a contact cut along $T_{-1}$. The resulting manifold is a solid torus $S$ whose meridian is a curve of slope $-1$ on $\partial S = \partial S_{\overline{0}}.$ In particular, the dividing set on $\partial S$ consists of two longitudinal curves. Since $S_{\overline{0}}$ was a contact submanifold of $S_{\epsilon}$, the solid torus $S$ is a contact submanifold of the manifold obtained from $S_{\epsilon}$ by admissible transverse $-1$-surgery on $K$. By Lemma \ref{lem:admutight}, this surgered manifold is a standard neighborhood of the dual core curve to $K$ and is therefore tight. Thus, $S$ is tight as well.

Next, observe that we can break $S_{\overline{0}}$ into $S'_{\overline{0}}\cup S_{\overline{0},\overline{0}}$, where $S'_{\overline{0}}$ is a contactomorphic to $S_{\overline{0}}$ and $S_{\overline{0},\overline{0}}$ is an $I$-invariant collar neighborhood of $\partial S_{\overline{0}}$ in $S_{\overline{0}}$. Legendrian surgery on $L$ is performed by removing $S'_{\overline{0}}$ from $S_{\overline{0}}$ and re-gluing according to the map described in Subsection \ref{ssec:leg}. The result is again a tight solid torus $S'$ whose meridian is  a curve of slope $-1$ on $\partial S' = \partial S_{\overline{0}}.$ Theorem \ref{Kanda} then implies that $S$ is contactomorphic to $S'$.
\end{proof}

\subsection{Admissible Transverse Surgeries Which Are Not Local Legendrian Surgeries}
\label{ssec:admnotleg}
In the previous subsection, we saw that certain admissible transverse surgeries can be realized via local Legendrian surgeries.  Here, we show that this is not true for all admissible transverse surgeries, per Theorem \ref{excludedsurgery}. This theorem clearly follows from the two propositions below. The first identifies the sequence $\{a_k\}$ described in Theorem \ref{excludedsurgery}.

Throughout this section we assume that $K$ is a transverse knot in some contact manifold and $N$ is a standard neighborhood of $K$ such that the characteristic foliation on $\partial N$ is linear with slope $a$, where $n<a<n+1$ for some integer $n$. 

\begin{proposition}
\label{prop:admnotleg1}
Suppose $n<a<n+1$ for some integer $n$, and let $a_k=\frac{k(n+1)+n}{k+1}$ for any non-negative integer $k$. Then, for $a_k<a$, the contact manifold obtained from $S_a$ by admissible transverse $a_k$-surgery on its core $K$ is not the result of Legendrian surgery on any Legendrian link in $S_a$.
\end{proposition}

\begin{proof}
We can identify $S_a$ with a standard neighborhood of a transverse unknot in the tight contact structure on $S^3$ such that $n=-1$ when $S_a$ is framed by the Seifert framing on the unknot. Note that $a_k=\frac{-1}{k+1}$ with respect to this embedding. It follows that, for $a_k<a$, admissible transverse $a_k$-surgery on the core $K$ of $S_a$ recovers the tight contact structure on $S^3$. To see this, simply note that this surgery is the union of the universally tight solid torus $\overline{S^3-S_a}$ with the universally tight solid torus obtained from $S_a$ by this transverse surgery. Now, if this admissible transverse surgery could also be achieved via Legendrian surgery on some link in $S_a$, then, by attaching 2-handles to $B^4$ along this link, we would obtain a Stein filling of $(S^3,\xi_{std})$ with non-trivial second homology. But Gromov showed in \cite{Gromov85} that $B^4$ is the unique Stein filling of $(S^3,\xi_{std})$.
\end{proof}

The proposition below identifies the sequence $\{d_k\}$ described in Theorem \ref{excludedsurgery}.

\begin{proposition}
\label{prop:admnotleg2}
Suppose $n<a<n+1$ for some integer $n$, and let $d_k =\frac{(k+1)n+1}{k+1}$ for any positive integer $k$.  Then, for $d_k<a$ and $k\neq 3$, the contact manifold obtained from $S_a$ by admissible transverse $d_k$-surgery on its core $K$ is not the result of Legendrian surgery on any Legendrian link in $S_a$.
\end{proposition}

\begin{proof}
As in the proof of Proposition \ref{prop:admnotleg1}, we identify $S_a$ with a standard neighborhood of a transverse unknot in the tight contact structure on $S^3$ such that $n=-1$ when $S_a$ is framed by the Seifert framing on the unknot. Note that $d_k=\frac{-k}{k+1}$ with respect to this embedding. Then, for $d_k<a$, admissible transverse $d_k$-surgery on the core $K$ of $S_a$ results in $L(k,1)$ with some contact structure. 

Now, if this admissible transverse surgery could be achieved via Legendrian surgery on some link $L\subset S_a$, then, by attaching 2-handles to $B^4$ along $L$, we would obtain a Stein manifold with boundary $L(k,1)$. Then there is a unique Stein filling of $L(k,1)$ if $k\not=4$ and its second homology has rank one \cite{McDuff90, plavhm}, which implies that the link $L$ is a knot. 
Moreover there are two Stein fillings of $L(4,1)$, one of which has second homology of rank one and the other is a rational homology ball \cite{McDuff90}, so once again the only possibility for $L$ is that it is a knot. 
In particular, $L$ is a knot in the solid torus $S_a$ on which some integral surgery yields another solid torus. 

Integral surgeries on knots $L$ in a solid torus $S$ which yield another solid torus $S'$ were classified in \cite{Berge91, Gabai89}.  According to this classification either (1) $L$ is the core of $S$ or (2) the meridian of $S'$ is a curve on $\partial S$ of slope $p/q$, where $q=B^2$ and $p=bB+\delta A$ for some integers $A,B,b,\delta$. Moreover, these integers must satisfy $0<2A\leq B$, $(A,B)=1$ and $\delta=\pm1$ (along with some other conditions; see \cite[Theorem 2.5]{Berge91}). 

Returning to our situation, case (1) is not possible since $L$ is supposed to be some knot in $S_a$ on which some integral surgery agrees with $\frac{-k}{k+1}$-surgery on the core of $S_a$; this would imply that $\frac{-k}{k+1}$ is an integer, which is not true for any positive integer $k$. In case (2), we see that $k+1 = B^2$ and $-k=bB+\delta A$, from which it follows that $-\delta A +1$ is divisible by $B$. This is only possible, given the restrictions on $A,B,b,\delta$, if $A=1$, $B=2$, $b=-2$ and $\delta = 1$. Therefore, case (2) is not possible as long as $k\neq 3$. This completes the proof.
\end{proof}
\begin{remark}
Note that if $S_a$ can be thickened to the solid torus $S_{(2n+1)/2}$, then one can achieve admissible transverse $d_3$-surgery on $K$ via Legendrian surgery on a leaf in the foliation of $\partial S_{(2n+1)/2}$. So the admissible transverse $d_3$-surgery can be achieved by Legendrian surgery on a ``semi-local" Legendrian knot. 
\end{remark}

\subsection{Uniform Thickness and Surgery}
\label{ssec:uniformthick}
In this section, we prove Theorem~\ref{thm:utgood}, which states that every admissible transverse surgery on a transverse knot in a uniformly thick knot type can be achieved via a sequence of Legendrian surgeries in a neighborhood of the transverse knot.
We start with a preliminary lemma that does not require uniform thickness, but just some thickening. 

\begin{lemma}\label{lem:thicken}
Let $K$ be a transverse knot in some contact manifold. Suppose $N$ is a standard neighborhood of $K$ such that the characteristic foliation on $\partial N$ is linear with slope $a$, where $n<a<n+1$ for some integer $n$. If $N$ thickens to a standard neighborhood $S$ of a Legendrian knot with Thurston-Bennequin invariant $n$, then surgery of slope less than $n$ is an admissible transverse surgery on $K$ and can be achieved by Legendrian surgery on some link in $S$. 
\end{lemma}

\begin{proof}
From the models for the standard neighborhoods of Legendrian and transverse knots discussed in Section \ref{sec:prelim}, note that, for any $r<n$, we can find a standard neighborhood of $K$ contained in $S$ that is contactomorphic to $S_r$. Thus, any surgery on $K$ of slope less than $n$ is admissible. 

It is enough to show that the contact manifold obtained from $S$ via admissible transverse $r$-surgery on $K$ can also achieved by Legendrian surgery on a link in $S$. The former surgery results in a universally tight solid torus $S'$ whose meridian is a curve of slope $r$ on $\partial S' = \partial S$. According to Ding and Geiges \cite{DG2} (or the analysis performed in the proof of Theorem~\ref{SimpleSurg}), one can find a Legendrian link $\mathbb{L}$ in $S$ consisting of stabilizations of Legendrian pushoffs of $L'$ such that the manifold obtained from $S$ via Legendrian surgery on $\mathbb{L}$ is a solid torus $S''$ whose meridian is a curve of slope $r$ on $\partial S'' = \partial S$ for any $r<\tb(L')=n$. Moreover, by choosing the stabilizations for these pushoffs appropriately, we can achieve any tight contact structure on $S''$ with the given dividing set on the boundary. It follows that there is some Legendrian link $\mathbb{L}\subset S$ on which Legendrian surgery yields a solid torus $S''$ contactomorphic to $S'$, completing the proof.
\end{proof}

\begin{proof}[Proof of Theorem~\ref{thm:utgood}]
Suppose $K$ is a transverse knot in $(M,\xi)$ belonging to a uniformly thick knot type $\mathcal{K}$. Let $S_a$ be any standard neighborhood of $K$. It is enough to show that any admissible transverse $r$-surgery for $r<a$ can be achieved via Legendrian surgery in some neighborhood of $K$.  

The universal thickness hypothesis implies that $S_a$ is contained in a standard neighborhood $N$ of a Legendrian knot $L$ in the knot type $\mathcal{K}$, where $L$ has maximal Thurston-Bennequin invariant among all Legendrian representatives of $\mathcal{K}$. Let $n-1$ be the greatest integer less than or equal to $a$. So $n-1\leq a<n$. By the classification of tight contact structures on thickened tori \cite{Giroux00, honda2}, there is a torus $T\subset N-S_a$ which is convex and has two dividing curves of slope $n$. Let $S$ be the solid torus in $N$ bounded by $T$. By Theorem~\ref{Kanda}, $S$ is a standard neighborhood of a Legendrian knot $L'$ with $\tb(L')=n$. The theorem now follows from Lemma~\ref{lem:thicken}.
\end{proof}

\begin{proof}[Proof of Corollary \ref{cor:uthickt}]
This corollary follows immediately from the proof of Theorem \ref{thm:utgood} and the observation that $r$-surgery on $K$ is admissible for any $r<\tb(L)$. Indeed, if $L$ is the Legendrian representative of $K$ with maximal Thurston-Bennequin invariant, and $N$ is its standard neighborhood, then there is a standard neighborhood $S_a$ of $K$ embedded in $N$ for any $a<\tb(L)$.
\end{proof}

\section{Our examples}
\label{sec:examples}
In this section, we describe the open books which support the contact manifolds in our examples. First, we review some notions related to diffeomorphisms of surfaces.

Suppose $S$ is a compact, orientable surface with boundary and that $\phi$ is a diffeomorphism of $S$ which restricts to the identity on $\partial S$. In a slight abuse of notation, we call $\phi$ pseudo-Anosov if it is freely isotopic to a homeomorphism $\phi_0$ of $S$ which is pseudo-Anosov in the conventional sense: {\em i.e.}\/, there exist two singular measured foliations of $S$, $(\mathcal{F}_s,\mu_s)$ and $(\mathcal{F}_u,\mu_u)$, which are transverse, such that $\phi_0(\mathcal{F}_s,\mu_s)=(\mathcal{F}_s,\lambda\mu_s)$ and $\phi_0(\mathcal{F}_u,\mu_u)=(\mathcal{F}_u,\lambda^{-1}\mu_u)$ for some $\lambda>1$ \cite{th2}. 

The \emph{fractional Dehn twist coefficient} (FDTC) of $\phi$ around a boundary component $B$ of $S$ is defined as follows. Let $x_0,\dots, x_{n-1}$ be the attracting fixed points of $\phi_0$ on $B$, labeled in order as one traverses $B$ in the direction specified by its orientation.
Since $\phi_0$ permutes the points $\{x_i\}$, there exists an integer $k$ such that $\phi_0$ sends $x_i$ to $x_{i+k}$ for all $i$ modulo $n$. If $H:S\times [0,1]\rightarrow S$ is the free isotopy from $\phi$ to $\phi_0$, and $\beta:B\times [0,1]\rightarrow B\times [0,1]$ is the map defined by 
\[
\beta(x,t)=(H(x,t),t),
\] 
then $\beta(x_i \times [0,1])$ is an arc from $(x_i,0)$ to $(x_{i+k},1)$. The FDTC of $\phi$ around $B$ is the fraction $c \in \mathbb{Q}$, where $c\equiv k/n \text{ modulo }1$ is the number of times the arc $\beta(x_i\times [0,1])$ wraps around the cylinder $B\times [0,1]$.

A related notion is that of being \emph{right-veering}. Suppose $\alpha:[0,1]\rightarrow S$ is a properly embedded arc. Let $\beta$ be another such arc with $\beta(0) = \alpha(0)$. We write $\alpha \geq \beta$ if either $\alpha \sim \beta$ or if, after isotoping $\beta$ (while fixing its endpoints) so that it intersects $\alpha$ efficiently, $(\dot\beta(0),\dot\alpha(0))$ defines the orientation of $S$ at $\alpha(0) = \beta(0)$. The monodromy $\phi$ is said to be right-veering if $\alpha \geq \phi(\alpha)$ for all such $\alpha$. Honda, Kazez and Mati{\'c} prove in \cite{hkm1} that a contact structure is tight if and only if all of its supporting open books are right-veering. They also establish the following relationship.

\begin{proposition}[Honda Kazez and Mati\'c 2007, \cite{hkm1}]
A pseudo-Anosov map is right-veering if and only if its FDTCs are all positive.
\end{proposition}

Henceforth, $T$ shall denote the genus one surface with two boundary components, $B_1$ and $B_2$. Let $\psi$ be the diffeomorphism of $T$ given by the product of Dehn twists, 
\[
\psi = D_aD_b^{-1}D_cD_{d}^{-1},
\] 
where $a,b,c$ and $d$ are the curves shown in Figure \ref{fig:OB}. Then $\psi$ is pseudo-Anosov by a well-known construction of Penner \cite{penner}. (Penner showed that if $\mathcal{S}^+\cup \mathcal{S}^-$ is a collection of curves which fills a surface such that the curves in each of $\mathcal{S}^+$ and $\mathcal{S}^-$ are pairwise disjoint, then any factorization consisting of positive Dehn twists along the curves in $\mathcal{S}^+$ and negative Dehn twists along those in $\mathcal{S}^-$ is pseudo-Anosov as long as this factorization contains at least one Dehn twist along every curve in the collection.)

We define \[\psi_{n,k_1,k_2} = D_{\delta_1}^{k_1} D_{\delta_2}^{k_2} \cdot\psi^n,\] where $\delta_1$ and $\delta_2$ are curves parallel to the boundary components $B_1$ and $B_2$ of $T$. 

\begin{figure}[!htbp]

\labellist 
\hair 2pt 
\small
\pinlabel $d$ at -1 150
\pinlabel $a$ at 134 54
\pinlabel $b$ at 154 100
\pinlabel $c$ at 134 155
\pinlabel $\delta_1$ at -8 198
\pinlabel $\delta_2$ at 253 198

\pinlabel $\alpha_2$ at 415 155
\pinlabel $\alpha_1$ at 352 147
\tiny
\pinlabel $180^{\circ}$ at 149 14
\endlabellist 

\begin{center}
\includegraphics[width=10cm]{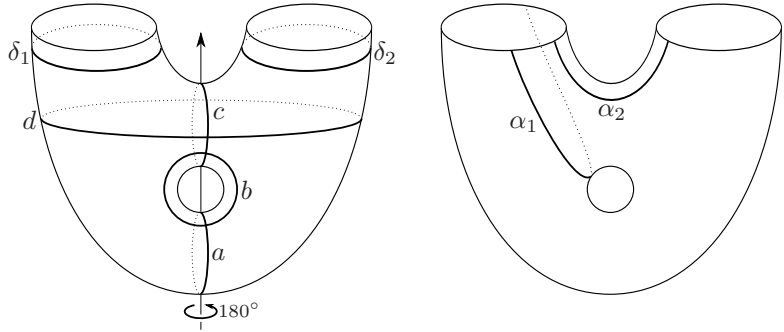}
\caption{\quad The surface $T$. The involution $\iota$ is a $180^{\circ}$ rotation about the axis shown on the left.}
\label{fig:OB}
\end{center}
\end{figure}

\begin{lemma}\label{lem:fdtc=n}
The FDTC of $\psi_{n,k_1,k_2}$ around $B_i$ is $k_i$.
\end{lemma}

\begin{proof}
Consider the rotation $\iota$ of $T$ by $180^{\circ}$ around the axis shown in Figure \ref{fig:OB}. Since $\iota$ exchanges $B_1$ and $B_2$ and commutes with $\psi$, the two FDTCs of $\psi$ are equal. Consider the arcs $\alpha_1$ and $\alpha_2$ shown in Figure \ref{fig:OB}. It is clear that $\psi(\alpha_1)\geq \alpha_1$; therefore, $\psi$ is not right-veering and its FDTCs are less than or equal to $0$. On the other hand, $\alpha_2\geq \psi(\alpha_2)$, which implies that $\psi^{-1}$ is not right-veering. The FDTCs of $\psi^{-1}$ are therefore less than or equal to $0$, which implies that those of $\psi$ are greater than or equal to $0$. Thus, the FDTCs of $\psi$ are $0$. The same is therefore true of $\psi^n$. The lemma follows.
 \end{proof}
 
We shall denote by $(M_{n,k_1,k_2},\xi_{n,k_1,k_2})$ the contact manifold supported by the open book $(T,\psi_{n,k_1,k_2})$. Then, we have the following result.

\begin{proposition}\label{prop:utight}
$(M_{n,k_1,k_2},\xi_{n,k_1,k_2})$ is universally tight if $k_1,k_2\geq 2$. \qed
\end{proposition}

This proposition follows from the more general result of Colin and Honda below.

\begin{theorem}[Colin and Honda 2008, \cite{ch}]
\label{thm:utightfdtc}
Suppose $\phi$ diffeomorphism of a surface $S$ with boundary components $B_1,\dots,B_k$ which is the identity on the boundary and is freely isotopic to a pseudo-Anosov diffeomorphism $\phi_0$, and that $\phi_0$ has $n_i$ attracting fixed points on the boundary $B_i$, for $i=1,\dots,k$.  If the FDTC of $\phi$ around $B_i$ is at least $2/n_i$ for $i=1,\dots,k$, then the contact manifold supported by $(S,\phi)$ is universally tight and the result of Legendrian surgery on any Legendrian link in this contact manifold is also tight.
\end{theorem}

Below, we describe how Theorem  \ref{thm:utightfdtc} follows from work of Colin and Honda in \cite{ch}. Their work involves Eliashberg and Hofer's {\em contact homology}. We will need only the following basic features of contact homology; for more details, see \cite{Eliashberg}.
\begin{enumerate}
\item To a closed contact manifold $(M,\xi)$ with a generic contact 1-form $\alpha$, one can associate a differential graded algebra $(\mathcal{A}(\alpha),\partial)$ whose homology $CH(M,\xi)$ depends only on $(M,\xi)$. This is the \emph{contact homology} of $(M,\xi)$.
\item Given an exact symplectic manifold $(X,d\lambda)$ with $\partial X= M_1\cup -M_2$ and $\alpha_1=\lambda|_{M_1}$ and $\alpha_2=\lambda|_{-M_2}$ positive contact forms on $M_1$ and $M_2$, respectively, 
 there is a chain map $\Phi_X:(\mathcal{A}(\alpha_1),\partial_1)\to (\mathcal{A}(\alpha_2), \partial_2)$ that sends $1$ to $1$,  \cite{Eliashberg}.
\item If $(M,\xi)$ is overtwisted, then $CH(M,\xi)=0$, \cite{yau}.
\end{enumerate}
Lastly, recall that an {\em augmentation} of a differential graded algebra $(\mathcal{A},\partial)$ is a chain map $\epsilon:(\mathcal{A},\partial)\to (\Z/2,0)$ that sends $1$ to $1$, where $(\Z/2, 0)$ is the trivial differential graded algebra.

\begin{proof}[Proof of Theorem \ref{thm:utightfdtc}]
Suppose $(M,\xi)$ is a contact manifold supported by an open book $(S,\phi)$ which satisfies the hypotheses in Theorem \ref{thm:utightfdtc}. In \cite[Theorem 2.3 (1)]{ch}, Colin and Honda show that the differential graded algebra $(\mathcal{A}(\alpha),\partial)$ for the contact homology of $(M,\xi)$ admits an augmentation for any  generic contact 1-form $\alpha$ defining $\xi$. (In \cite{ch}, this theorem is stated for open books with connected binding, but this is only for notational convenience; their proof works just as well for open books with multiple binding components.) It follows from this theorem that $CH(M,\xi)$ surjects onto $\Z/2$. In particular, $CH(M,\xi)$ is non-trivial, which implies that $(M,\xi)$ is tight by property (3) above.

The fact that 3-manifold groups are residually finite implies that if the universal cover is overtwisted then so is some finite cover, as detailed in \cite{hkm5}. So, to prove the universal tightness statement in Theorem \ref{thm:utightfdtc}, Colin and Honda need only show that all finite covers are tight. But, for any finite cover, it is easy to see that the open book $(S,\phi)$ pulls back to a compatible open book which also satisfies the hypotheses of Theorem \ref{thm:utightfdtc}. This establishes the first statement of the theorem. 

For the statement about Legendrian surgery, suppose that $(M',\xi')$ is obtained from $(M,\xi)$ by Legendrian surgery on some Legendrian link. Property (2) implies that the associated Stein cobordism $X$ from $M$ to $M'$ induces a nontrivial chain map $\Phi_X: (\mathcal{A}(\alpha'),\partial')\to (\mathcal{A}(\alpha),\partial)$, where $\alpha$ and $\alpha'$ are the 1-forms coming from the exact symplectic form on $X$. Composing this chain map with the augmentation of $(\mathcal{A}(\alpha),\partial)$, we obtain an augmentation of $(\mathcal{A}(\alpha'),\partial')$, which implies that $\xi'$ is not overtwisted.
\end{proof}

We now return to our examples.

\begin{proof}[Proof of Proposition \ref{prop:utight}]
Proposition \ref{prop:utight} follows immediately from Theorem \ref{thm:utightfdtc} even without determining the attracting fixed points  on $\partial T$ of the pseudo-Anosov diffeomorphisms isotopic to $\psi_{n,k_1,k_2}$.  
That said, it is easy to check, by examining the train track for $\psi_n$ (constructed as in \cite{penner}), that each of $B_1$ and $B_2$ contains exactly one such fixed point. 
\end{proof}

Below, we establish a couple additional properties of the manifolds $M_{n,k_1,k_2}$. 

\begin{lemma}\label{lem:qhs3}
For each $n\not\equiv 0\text{ modulo }2$, $M_{n,k_1,k_2}$ is a rational homology 3-sphere for all but finitely many values of $k_1+k_2$.
\end{lemma}

\begin{proof}
Let $|H_1(M_{n,k_1,k_2};\mathbb{Z})|$ denote the order of $H_1(M_{n,k_1,k_2};\mathbb{Z})$ when this group is finite, and zero when it is infinite. It is easy to see (e.g., from a surgery presentation) that, after fixing $n$ and $k_1$, this integer is a polynomial function of $k_2$; in fact, it depends only on $n$ and the sum $k_1+k_2$ (composing with $D_{\delta_1}D_{\delta_2}^{-1}$ does not affect the first homology since $\delta_1$ and $\delta_2$ are homologous in $T$). Therefore, for a fixed $n$, $|H_1(M_{n,k_1,k_2};\mathbb{Z})|$ is a polynomial function of $k_1+k_2$. To prove the lemma, it thus suffices to show that $|H_1(M_{n,k_1,k_2};\mathbb{Z})|$ is non-zero for some $k_1$ and $k_2$ whenever $n\not\equiv 0\text{ modulo }2$.

Observe that $\psi_{n,0,0}$ commutes with the rotation $\iota$ described above. Since $T/\iota\cong D^2$ and $\iota$ fixes 4 points on $T$, the manifold $M_{n,0,0}$ is the double cover of $S^3$ branched along some closed 4-braid $\beta_n$. Note that $D_d$ is isotopic to $(D_aD_b)^6$; therefore, $\psi_{n,0,0}$ may be expressed as \[( D_aD_b^{-1}D_c\cdot(D_aD_b)^{-6})^n,\] which implies that $\beta_n$ is the closure of the braid given by \[(  \sigma_1\sigma_2^{-1}\sigma_3\cdot (\sigma_1\sigma_2)^{-6})^n.\] It is easy to check that $\beta_n$ is a knot for $n\not\equiv 0\text{ modulo 2}$; hence, $|H_1(M_{n,0,0};\mathbb{Z})|$ is non-zero (indeed, odd) for these values of $n$.
\end{proof}

The following oddly phrased lemma will be useful in the next section.

\begin{lemma}\label{lem:hyperbolic}
For each $n\geq 14$, there exist infinitely many pairs $(k_1,k_2)$ with $2\leq k_2\leq n$ such that $M_{n,k_1,k_2}$ is hyperbolic. 
\end{lemma}

\begin{proof}
Since $\psi_{n,0,0}$ is pseudo-Anosov, its mapping torus is hyperbolic. The manifold $M_{n,k_1,k_2}$ is obtained from this mapping cylinder by filling the cusps corresponding to $B_1$ and $B_2$ along the slopes $-1/k_1$ and $-1/k_2$, respectively. According to Thurston's Dehn Surgery Theorem, all but finitely many fillings of the first cusp are hyperbolic \cite{th1}. Moreover, for each hyperbolic filling of the first cusp, all but at most 12 fillings of the second cusp are hyperbolic as well \cite{agol}. The lemma follows.
\end{proof}

\section{Capping off open books}
\label{sec:cappingoff}

Let $(S,\phi)$ be an open book with at least two binding components, and let $(\widehat S,\widehat \phi)$ denote the open book obtained by capping off one of the boundary components of $S$, as in the Introduction. In this section, we use the main result of \cite{bald5} to prove Theorems \ref{otcapoff}, \ref{thm:atoroidalcneq0} and \ref{thm:fdtc} and Corollary \ref{cor:atoroidaltight}. We then show that capping off can be viewed as admissible transverse surgery, proving Theorems \ref{thm:main} and \ref{thm:cappingsurgery}.

\subsection{The Map on Heegaard Floer Homology}
\label{ssec:hfmap}
Let $(M_{S,\phi},\xi_{S,\phi})$ denote the contact 3-manifold supported by the open book $(S,\phi)$. Consider the cobordism $W$ from $M_{S,\phi}$ to $M_{\widehat S,\widehat\phi}$ obtained by attaching a $0$-framed 2-handle along the binding component in $M_{S,\phi}$ corresponding to the capped off boundary component of $S$. Given this set up we have the following result.

\begin{theorem}[\rm{\cite[Theorem 1.2]{bald5}}]
\label{thm:nat}
There exists a $\Sc$ structure $\spc$ on $W$ such that \[F_{W,\spc}:\hf(-M_{\widehat S,\widehat\phi})\rightarrow \hf(-M_{S,\phi})\] sends $c(\widehat S,\widehat \phi)$ to $c(S,\phi).$
\end{theorem}

It follows immediately that $c(S,\phi)$ vanishes whenever $c(\widehat S,\widehat \phi)$ does. This prompted Question~1.4 from \cite{bald5}, which asks whether $\xi_{S,\phi}$ is overtwisted whenever $\xi_{\widehat S, \widehat \phi}$ is. Our examples show that the answer is ``no". 
\begin{proof}[Proof of Theorem~\ref{otcapoff}]
Consider the open book $(T,\psi_{n,k_1,k_2})$, where $2\leq k_2\leq n$ and $2\leq k_1$. By Proposition \ref{prop:utight}, $\xi_{n,k_1,k_2}$ is universally tight. Now, cap off the boundary component $B_1$ of $T$. The resulting surface $\widehat T$ has genus one and a single boundary component, and the induced monodromy is the product \[\widehat \psi_{n,k_1,k_2}=D_{\delta}^{k_2-n}\cdot (D_x^2D_y^{-1})^n,\] where $x$, $y$ and $\delta$ are the curves on $\widehat T$ shown in Figure \ref{fig:OB2}. The contact structure supported by $(\widehat T, \widehat \psi_{n,k_1,k_2})$ is overtwisted since the FDTC of $\widehat \psi_{n,k_1,k_2}$ is given by $k_2-n\leq 0$. 
\end{proof}
\begin{figure}[!htbp]

\labellist 
\hair 2pt 
\small
\pinlabel $\delta$ at -6 137
\pinlabel $x$ at 93 25
\pinlabel $y$ at 113 63

\endlabellist 

\begin{center}
\includegraphics[width=3.5cm]{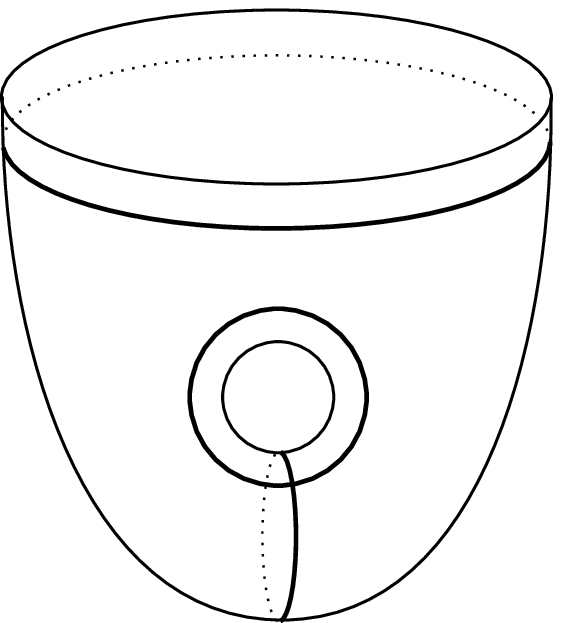}
\caption{\quad The capped off surface $\widehat T$. }
\label{fig:OB2}
\end{center}
\end{figure}

In Remark \ref{rmk:contacthomology}, we commented that the examples considered in the proof of Theorem \ref{otcapoff} can be used to provide examples of contact manifolds with non-vanishing contact homologies which are strongly symplectically cobordant to overtwisted contact manifolds. We justify this below.

\begin{proof}[Justification of Remark \ref{rmk:contacthomology}]
In \cite{gay}, Gay showed that the 2-handle cobordism $W$ from $M_{S,\phi}$ to $M_{\widehat S, \widehat \phi}$ can be equipped with a symplectic form $\omega$, with respect to which $(M_{S,\phi},\xi_{S,\phi})$ is strongly concave and $(M_{\widehat S,\widehat\phi},\xi_{\widehat S,\widehat\phi})$ is weakly convex. Now, suppose $(S,\phi)$ and $(\widehat S,\widehat \phi)$ are the open books $(T,\psi_{n,k_1,k_2})$ and $(\widehat T, \widehat \psi_{n,k_1,k_2})$ in the proof of Theorem \ref{otcapoff}. The manifold $M_{\widehat S, \widehat \phi}$ is a rational homology 3-sphere in this case ({\em cf.}\/ \cite{bald1}), and so the boundary component $(M_{\widehat S,\widehat\phi},\xi_{\widehat S,\widehat\phi})$ of $(W,\omega)$ can be made to be strongly convex \cite{oono}. On the other hand, the proof of Theorem \ref{otcapoff} shows that $\xi_{\widehat S,\widehat\phi}$ is overtwisted. In other words, $(W,\omega)$ is a strong symplectic cobordism from a contact structure with non-trivial contact homology to an overtwisted contact structure. Note that, by the second property of contact homology listed in the previous subsection, $(X,\omega)$ cannot be an exact symplectic cobordism from $M_{S,\phi}$ to $M_{\widehat S, \widehat \phi}$. The non-exactness of this kind of symplectic cobordism was previously demonstrated in \cite{wen2}.
\end{proof}

The proof of Theorem \ref{otcapoff}, together with Theorem \ref{thm:nat} and Lemma \ref{lem:hyperbolic}, implies the following more precise version of Theorem \ref{thm:atoroidalcneq0}.

\begin{proposition}\label{prop:nvanishingc}
For each $n\geq 14$, there exist infinitely many pairs $(k_1,k_2)$ with $2\leq k_2\leq n$ such that $M_{n,k_1,k_2}$ is hyperbolic, $\xi_{n,k_1,k_2}$ is universally tight and $c(\xi_{n,k_1,k_2})=0.$ \qed
\end{proposition}

\begin{proof}[Proof of Corollary \ref{cor:atoroidaltight}]
Note that the contact manifolds in Proposition \ref{prop:nvanishingc} are strongly non-fillable since their contact invariants vanish \cite{ghiggini}. Lemma \ref{lem:qhs3} implies that infinitely many of these are rational homology 3-spheres and are therefore non-weakly fillable as well \cite{oono}. 
\end{proof}

We end this subsection with the proof of Theorem \ref{thm:fdtc}.

\begin{proof}[Proof of Theorem \ref{thm:fdtc}]
Lemma \ref{lem:qhs3} implies that for any $N>0$, there exist $n,k_1>N$ such that $M_{n,k_1,n}$ is a rational homology 3-sphere. According to the discussion above, $c(\xi_{n,k_1,n})=0.$ This implies that $\xi_{n,k_1,n}$ is non-weakly fillable and, hence, not the deformation of a co-orientable taut foliation \cite{yasha2, ethur, Etnyre04a}. Moreover, the FDTCs of $\psi_{n,k_1,n}$ around $B_1$ and $B_2$ are $k_1$ and $n$, by Lemma~\ref{lem:fdtc=n}.
\end{proof}

\subsection{Admissible Transverse Surgeries and Capping Off Open Books}
\label{ssec:admissiblecap}

In this subsection, we relate the contact structures obtained by capping off binding components of open books and those obtained by admissible transverse surgeries on these binding components. We first observe that if there is more than one binding component, then $0$-surgery on any single binding component is admissible, as implied by the lemma below.

\begin{lemma}\label{lem:nbhdB}
Suppose $(S, \phi)$ is an open book decomposition for $(M,\xi)$ with more than one binding component, and let $S_s$ be the contact solid torus defined in Subsection \ref{ssec:adm}. Then every binding component of $(S,\phi)$ has a neighborhood $N$ which is contactomorphic to $S_s$ for some $s\in (0,1)$, via a contactomorphism which sends the longitude on $\partial N$ induced by a page of the open book to the preferred longitude on $\partial S_s$. 
\end{lemma}

\begin{proof}
Denote the boundary components of $S$ by $B_1,\dots,B_n$. Let $S'$ be the surface obtained from $S$ by removing a collar neighborhood of $B_1$ on which $\phi$ is the identity, so that $S'$ has boundary components $B_1',B_2,\dots,B_n$. Let $\omega$ be a volume form on $S'$, and let $f:S'\rightarrow \R$ be a Morse function which achieves a maximum along $B_1'$, minima along the other $B_i$, and has no interior minima. Let $v=-\nabla f$ be the negative gradient flow of $f$. We can arrange that $v$ has divergence 1 for some volume form $\omega$; {\em i.e.}, that $\mathcal L_v(\omega) = \omega$. We can also arrange that there is some collar neighborhood $N_1'=(-1,-1/2]\times(0,2\pi]$ of $B_1'$ with coordinates $(s, \theta)$ such that $\omega  = ds \wedge d\theta$ and $v = s\frac{\partial}{\partial s}$. Likewise, we can arrange that each $B_i$ for $i=2,\dots,n$ has a collar neighborhood of the form $N_i=(1/2,1]\times(0,2\pi]$ on which $\omega  = ds \wedge d\theta$ and $v = s\frac{\partial}{\partial s}$. Note that $v$ points into $S$ along $B_1'$ and out of $S$ along the other $B_i$. Let $\lambda$ be the contraction $\lambda=i_{v}\omega$. The fact that $v$ has divergence 1 implies that $d\lambda = \omega$. Moreover, $\lambda = sd\theta$ with respect to the $(s,\theta)$ coordinates in the neighborhoods $N_i$. 

Let $A$ be the annulus $A = [-1/2,1]\times (0,2\pi]$. We can think of $S$ as the surface obtained by gluing $A$ to $S'$ along $B_1'=\{-1/2\}\times(0,2\pi]$, and we can think of $B_1$ as the new boundary component $\{1\}\times(0,2\pi]$. This component has collar neighborhood $N_1=N_1'\cup A = (-1/2,1]\times (0,2\pi]$ with coordinates $(s,\theta)$. We extend $\omega$ and $\lambda$ to $N_1$ by $\omega = ds \wedge d\theta$ and $\lambda  = sd\theta$. Note that $\omega$ and $\lambda$ behave near $B_1$ just as they do near the other $B_i$.

Consider the mapping torus $T_{\phi} = S\times[0,1]/(x,1)\sim (\phi(x),0)$. Its boundary consists of the tori $T_i = B_i\times S^1$ for $i=1,\dots,n$. We choose a longitude-meridian coordinate system on each $-T_i$, where the longitude is given by the intersection of $-T_i$ with a fiber $S\times \{t\}$, and the meridian is given by $\{p\}\times S^1$ for some $p\in B_i$. Recall that the manifold $M_{S,\phi}$ is gotten by Dehn filling each boundary component $T_i$ of $T_{\phi}$ with a solid torus $S^1\times D^2$ in such a way that the meridian of $S^1\times D^2$ is glued to the meridian on $-T_i$. 

Consider the 1-form on $T_{\phi}$ defined by $\alpha_K=Kdt+t\lambda+(1-t)\phi^*\lambda$. The fact that $d\lambda$ is a volume form on $S$ implies that $\alpha_K$ is a contact form for $K>0$ large enough. Let $\xi_K$ denote the contact structure on $T_{\phi}$ defined by this contact form. The local behavior of $\lambda$ near the $B_i$ implies that the characteristic foliation of $\xi_K$ on $-T_i$ has slope $s_i<0$. We may therefore extend $\xi_K$ to a contact structure $\xi_{S,\phi}$ on all of $M_{S,\phi}$ by Dehn-filling each $T_i$ as above, using the contact solid tori $S_{s_i}$. The binding components of $(S,\phi)$ are then the core curves $C_i\subset S_{s_i}$. Moreover, our construction of $\xi_{S,\phi}$ is equivalent to the standard construction of the contact structure compatible with the open book $(S,\phi)$. 

To complete the proof of Lemma \ref{lem:nbhdB}, simply observe that $A\times[0,1]/\sim$ is a standard contact thickened torus $S_{s_1,r_1}$, where $r_1>0$. In particular, $S_{r_1}$ is a standard neighborhood of $C_1$. 
\end{proof}

Recall that Theorem \ref{thm:cappingsurgery} states that if an open book $(S,\phi)$ has more than one binding component, then the contact manifold obtained via admissible transverse $0$-surgery on that component (which makes sense, thanks to Lemma \ref{lem:nbhdB}) is supported by the open book $(\widehat S, \widehat \phi)$ obtained from $(S,\phi)$ by capping off the corresponding boundary component of $S$.

\begin{proof}[Proof of Theorem~\ref{thm:cappingsurgery}]
We will use the notation from the proof of Lemma \ref{lem:nbhdB}. Let $\widehat S$ be the surface obtained by capping off $B_1$ with a disk $D$. Let $D'$ be a slightly larger disk in $\widehat S$ containing $D$ such that $\partial D' = \{-1/4\}\times (0,2\pi]\subset A$. Let $(r,\Theta)$ be the standard polar coordinate system on $D'$ for $0\leq r < 9/4$ and $\Theta\in (0,2\pi]$. Where $D'$ and $A$ overlap, their coordinates are related by $s = 2-r$ and $\theta = -\Theta$. Let $\widehat\omega$ and $\widehat\lambda$ be forms on $S$ which agree with the forms $\omega$ and $\lambda$ for $r\geq 17/8$. In particular, in the region $17/8 \leq r \leq 9/4$, which corresponds to $-1/4 \leq s \leq -1/8$, we have $\widehat\omega = ds\wedge d\theta = dr\wedge d\Theta$ and $\widehat\lambda = sd\theta = (r-2)d\Theta$. We extend these forms across the rest of the disk $D'$ by $\widehat\omega = f'(r)dr\wedge d\Theta$ and $\widehat\lambda = f(r)d\Theta$, where $f(r)$ is a smooth, increasing function of $r$ with 
\[f(r) = \begin{cases}
r-2, & 17/8 \leq r \leq 9/4\\
r^2 & r \leq 1/8.
\end{cases}
\]

The manifold $M_{\widehat{S},\widehat{\phi}}$ is then formed from the mapping torus $T_{\widehat{\phi}} = \widehat{S}\times[0,1]/(x,1)\sim (\phi(x),0)$ by Dehn-filling the boundary tori $T_2,\dots,T_n$ as before. The 1-form on $T_{\widehat{\phi}}$ defined by $\widehat{\alpha}_K=Kdt+t\widehat\lambda+(1-t)\widehat\phi^*\widehat\lambda$ is still a contact form for $K$ large enough and, so, defines a contact structure $\widehat\xi_K$. The contact structure $\xi_{\widehat S, \widehat \phi}$ on $M_{\widehat{S},\widehat{\phi}}$ compatible with $(\widehat S, \widehat \phi)$ is then formed from $\widehat\xi_K$ by Dehn-filling the $T_i$ using the contact solid tori $S_{s_i}$ as above. 

Now, let $(M',\xi')$ denote the result of admissible transverse $0$-surgery on the binding $C_1$ of $(S,\phi)$. Our goal is to show that this is contactomorphic to $(M_{\widehat{S},\widehat{\phi}},\xi_{\widehat S, \widehat \phi}).$ Let $S_{r_1}'$ be the contact solid torus in $M'$ obtained from $S_{r_1}$ by this surgery, and let $C_1'\subset S_{r_1}'$ be the dual transverse curve to $C_1$. As discussed in Subsection \ref{ssec:adm}, $S_{r_1}'$ is a standard neighborhood of $C_1'$. Let $U$ be the solid torus $(A\cup D')\times[0,1]/\sim$ in $M_{\widehat S, \widehat\phi}$. By construction, $M_{\widehat S, \widehat \phi}-U = M' - S_{r_1}'$, and the restriction of $\xi_{\widehat S, \widehat \phi}$ to $M_{\widehat S, \widehat \phi}-U$ agrees with the restriction of $\xi'$ to $M' - S_{r_1}'$. The proof is finished by the observation that $U$ is contactomorphic to $S_{r_1}'$, as both are universally tight solid tori with the same characteristic foliations on their pre-Lagrangian boundaries. Moreover, $U$ is pretty clearly a standard neighborhood of the transverse knot $e\times [0,1]/\sim$, where $e$ is the origin of $D'$.
\end{proof}

\begin{proof}[Proof of Theorem~\ref{thm:main}]
As mentioned in the Introduction, our main theorem, which states that admissible transverse surgery does not necessarily preserve tightness, follows immediately from the combination of Theorems~\ref{otcapoff} and ~\ref{thm:cappingsurgery}. 
\end{proof}

We end this subsection and our paper with the proof of Theorem \ref{prop:legadmintervals}, which states that the range of tight, admissible transverse surgeries on a transverse knot can be non-closed and disconnected.

\begin{proof}[Proof of Theorem~\ref{prop:legadmintervals}]
Let $(M_{n,k_1,k_2},\xi_{n,k_1,k_2})$ and $(T,\psi_{n,k_1,k_2})$ be the contact manifolds and open books considered in the proof of Theorem \ref{otcapoff}. In that proof, we saw that $(M_{n,k_1,k_2},\xi_{n,k_1,k_2})$ was universally tight, but the contact manifold obtained by capping off the boundary component $B_1$ was overtwisted. From Lemma~\ref{lem:nbhdB} we know that the corresponding binding component $C_1$ has a standard neighborhood contactomorphic to $S_a$ for some $0<a<1.$ From Theorems~\ref{SimpleSurg} and \ref{thm:extraslopes}, we know that admissible transverse $s$-surgery on $C_1$ can be achieved via Legendrian surgery for all $s<0$ and for a sequence of $0<s<a$ converging to $0$. These surgeries are therefore tight, by Theorem \ref{thm:utightfdtc}. On the other hand, since capping off is equivalent to admissible transverse $0$-surgery, by Theorem \ref{thm:cappingsurgery}, and the capped off contact manifold is overtwisted, $t(K)$ does not contain $0$. Thus, $t(K)$ is non-closed and disconnected. 
\end{proof}


\end{document}